\documentclass[12pt]{amsart}
\usepackage{amsthm,amssymb}
\usepackage{hyperref}
\usepackage{epsfig}
\usepackage{pstricks}
\usepackage{pst-plot}
\usepackage{tabmac}

\usepackage[margin=1.1in]{geometry}

\newtheorem{theorem}{Theorem}[section]
\newtheorem{thm}[theorem]{Theorem}

\newtheorem{lem}[theorem]{Lemma}

\newtheorem{prop}[theorem]{Proposition}

\theoremstyle{remark}
\newtheorem{example}{Example}
\newtheorem{remark}{Remark}

\def\Z{{\mathbb Z}}
\def\det{{\mathrm{det}}}
\def\LSym{\mathrm{LSym}}

\def\ind{\mathrm{ind}}
\def\T{{\mathcal T}}

\def\H{{\overline H}}
\def\D{{\overline D}}
\def\HH{{\overline {\mathtt H}}}
\def\DD{{\overline {\mathtt D}}}
\def\c{{\mathfrak C}}

\def\k{{ \kappa}}
\def\ok{{\overline \kappa}}
\def\y{{\bar x}}
\def\x{{\mathtt x}}
\def\z{{\bar {\mathtt x}}}
\def\uqsln{{U_q'({\mathfrak {\hat {sl_n}}})}}
\def\trop{{\rm trop}}
\title{Intrinsic energy is a loop Schur function}
 \author{Thomas Lam}\address
 {Department of Mathematics\\ University of Michigan\\ Ann Arbor\\ MI 48109 USA.}
 \date{\today}
 \email{tfylam@umich.edu}
 \urladdr{http://www.math.lsa.umich.edu/\~{ }tfylam}
 \thanks{T.L. was supported by NSF grant DMS-0652641 and DMS-0901111, and by a Sloan Fellowship.}
 \author{Pavlo Pylyavskyy}\address
{Department of Mathematics\\ University of Michigan\\ Ann Arbor\\ MI 48109 USA.}
 \email{pavlo@umich.edu}
 \urladdr{http://sites.google.com/site/pylyavskyy/}
 \thanks{P.P. was supported by NSF grant DMS-0757165.}
\begin{document}
\begin{abstract}
We give an explicit subtraction-free formula for the energy function in tensor products of Kirillov-Reshetikhin crystals for symmetric powers of the standard representation of $\uqsln$.  The energy function is shown to be the tropicalization of a stretched staircase shape loop Schur function. The latter were introduced by the authors in the study of total positivity in loop groups.
\end{abstract}
\maketitle
\section{Introduction}
The intrinsic energy function plays an important role in 
the path model for affine highest weight crystals \cite{KKMMNN}. 
The energy function is also related to charge statistic of Lascoux-Sch\"utzenberger on semistandard tableaux (see \cite{NY2}), which establishes a relation between one dimensional configuration sums arising in solvable lattice models and Kostka-Foulkes polynomials, cf. \cite{DJKMO, KKMMNN, KKMMNN2}. 

Let $B = B_1 \otimes \cdots \otimes B_m$ be a tensor product of $\uqsln$ Kirillov-Reshetikhin crystals, where each
$B_i$ is the crystal for a symmetric power of the standard representation.  We identify $B_i$ with the semistandard 
Young tableaux with row shape, filled with the numbers $1,2,\ldots,n$.  Let $b = b_1 \otimes \cdots \otimes b_m \in B$, and write $x_i^{(r+i-1)}$ for the number of $r$'s in $b_i$.  The upper index $(r+i-1)$ is to be considered as an element of $\Z/n\Z$.
Our main result is the following formula for the intrinsic energy function $\D_B$ of $B$.

Let $\delta_{t} = (t,t-1,\ldots,1)$ denote the staircase shape of side-length $t$.
\begin{theorem}\label{thm:main}
We have
$$
\D_B(b) = \min_T\left\{ \sum_{(i,j) \in (n-1)\delta_{m-1}} x_{T(i,j)}^{(i-j)}\right\},
$$
where the minimum is over all semistandard tableaux $T$ of shape $(n-1)\delta_{m-1}$, and entries in $1,2,\ldots,m$.
\end{theorem}

In the physical interpretation, each $b_i$ represents a particle, and the intrinsic energy function $\D_B(b)$ is defined as the sum of ${m \choose 2}$ local energies of interactions of particles.  Theorem \ref{thm:main} thus has the following interpretation: each tableau $T$ encodes a way for $m$ particles to interact {\it simultaneously}, and intrinsic energy is equal to the minimum of these.

In \cite{LP}, motivated by the study of total positivity for loop groups, we introduced a generalization of the ring of symmetric functions, called {\it loop symmetric functions} and denoted $\LSym$.  In particular, we defined distinguished elements of $\LSym$ called loop Schur functions (see Section \ref{sec:LSym}). It is shown in \cite{LP} that the algebra homomorphisms from $\LSym$ to $\mathbb R$ taking nonnegative values on (skew) loop Schur functions are in bijection with totally nonnegative elements of the formal loop group.

Recall that the tropicalization of a subtraction-free polynomial $f$, is obtained by replacing multiplication by addition, and replacing addition by taking minimums.  Theorem \ref{thm:main} is equivalent to

\begin{theorem}\label{thm:main2}
The function $\D_B$ is the tropicalization of the loop Schur function $s^{(0)}_{(n-1)\delta_{m-1}}$ in the variables
$\{\x_i^{(s)}\}$.
\end{theorem}  
Theorem \ref{thm:main2} is an immediate consequence of Theorems \ref{thm:stair} and \ref{thm:energyprod} below.

Theorems \ref{thm:main} and \ref{thm:main2} are {\it canonical} in the sense that they correspond to the monomial expansion of a polynomial.  That a piecewise-linear expression for $\D_B(b)$ exists is already clear from the literature.  However, the fact a subtraction-free formula exists (or equivalently the rational version $\DD_B(b)$ of energy is a polynomial with positive coefficients) is not apparent from the definition of $\D_B(b)$, even though the latter takes nonnegative values.  (See also Remark \ref{rem:irred}.)

\begin{example}
 Let $n=2$ and $m=3$. Then \begin{align*}\D_B(b) &= \min(x_1^{(1)}+x_1^{(2)}+x_2^{(1)}, x_2^{(1)}+x_1^{(2)}+x_2^{(1)}, x_3^{(1)}+x_1^{(2)}+x_2^{(1)},  x_1^{(1)}+x_1^{(2)}+x_3^{(1)},\\ &x_2^{(1)}+x_1^{(2)}+x_3^{(1)}, x_3^{(1)}+x_1^{(2)}+x_3^{(1)}, x_2^{(1)}+x_2^{(2)}+x_3^{(1)}, x_3^{(1)}+x_2^{(2)}+x_3^{(1)})\end{align*}
corresponding to the following tableaux of shape $\delta_2=(2,1)$:
$$
\tableau[sY]{1&1\\2} \;\;\; \tableau[sY]{1&2\\2} \;\;\; \tableau[sY]{1&3\\2} \;\;\; \tableau[sY]{1&1\\3} \;\;\;
\tableau[sY]{1&2\\3} \;\;\; \tableau[sY]{1&3\\3} \;\;\; \tableau[sY]{2&2\\3} \;\;\; \tableau[sY]{2&3\\3}
$$
\end{example}

We use in our calculations a birational analogue of the {\it {combinatorial $R$-matrix}}.  It was previously studied by Kirillov \cite{K} in the context of the Robinson-Schensted algorithm, by Noumi-Yamada \cite{NY, Y} in the context of discrete Painl\'{e}ve systems, by Etingof \cite{E} in the context Yang-Baxter equations, by Berenstein-Kazhdan \cite{BK} in the context of geometric crystals, and by the authors \cite{LP} in the context of total positivity of loop groups.

\section{Loop symmetric functions}\label{sec:LSym}
Fix an integer $n > 1$ throughout.
\subsection{Loop Schur functions}
Let $\left(\x_i^{(r)}\right)_{1 \leq i \leq m, \; r \in \Z/n\Z}$ be a rectangular array of variables. We recall from \cite{LP} the definition of the ring of {\it {loop symmetric functions}}\footnote{In \cite{LP}, there are two such rings: the ring of whirl loop symmetric functions, and the ring of curl loop symmetric functions.  We use the former here.  Furthermore, we only use finitely many variables here.} in the variables $\x_i^{(r)}$, denoted $\LSym_m$.  A detailed study of loop symmetric functions will appear in \cite{LP2}.

For $k \geq 1$ and $r \in \Z/n\Z$, define the {\it {loop elementary symmetric functions}} and {\it {loop complete homogenous symmetric functions}} by
\begin{align*}
e_k^{(r)}(\x_1,\x_2,\ldots,\x_m) &= \sum_{1 \leq i_1 < i_2 < \cdots < i_k\leq m} \x_{i_1}^{(r)} \x_{i_2}^{(r+1)} \cdots \x_{i_k}^{(r+k-1)} \\
h_k^{(r)}(\x_1,\x_2,\ldots,\x_m) &= \sum_{1 \leq i_1 \leq i_2 \leq \cdots \leq i_k\leq m} \x_{i_1}^{(r)} \x_{i_2}^{(r-1)} \cdots \x_{i_k}^{(r-k+1)} 
\end{align*}
By convention, $e_k^{(r)}=h_k^{(r)} = 0$ for $k < 0$, and $e_0^{(r)}=h_0^{(r)}=1$.  Note that $e_k^{(r)} = 0$ for $k > m$. We call the upper index the {\it color}.  When all $n$ colors are identified, that is $\x_i^{(s)}=\x_i^{(s')}$ for all $i$ and $s,s'\in \Z/n\Z$, these functions specialise to the usual elementary and complete homogenous symmetric functions \cite{EC2}.  We define $\LSym_m$ to be the ring generated by the $e_k^{(r)}$.  Although it is not immidiately obvious, the $h_k^{(r)}$ lie in $\LSym_m$. In fact, both the $e_k^{(r)}$ and the $h_k^{(r)}$ are instances of distinguished elements of $\LSym_m$ called {\it {loop Schur functions}}.  

A square $s = (i,j)$ in the $i$-th row and $j$-th column has {\it content} $c(s)=i-j$.  We caution that our notion of content is the negative of the usual one.  Let $\rho/\nu$ be a skew shape.   Recall that a semistandard Young tableaux $T$ with shape
$\rho/\nu$ is a filling of each square $s \in \rho/\nu$ with an
integer $T(s) \in \Z_{> 0}$ so that the rows are weakly-increasing, and columns are increasing.  For $r\in \Z/n\Z$, the $r$-weight $\x^T$ of a tableaux $T$ is given by $\x^T = \prod_{s \in \rho/\nu}\x_{T(s)}^{(c(s)+r)}$.

We shall draw our shapes and tableaux in English notation:
$$
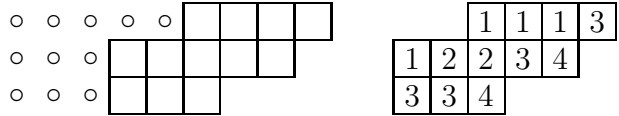
\tableau[sY]{\bl \circ&\bl \circ&\bl \circ&\bl \circ&\bl \circ&&&&\\\bl \circ&\bl \circ&\bl \circ&&&&&\\\bl \circ&\bl \circ&\bl \circ&&&} \qquad
\tableau[sY]{\bl &\bl&1&1&1&3\\1&2&2&3&4\\3&3&4}
$$
%

For $n = 3$ the $0$-weight of the above tableau is $(\x_1^{(1)})^2 (\x_3^{(1)})^3
\x_1^{(2)} \x_2^{(2)} \x_3^{(2)} \x_1^{(3)} \x_2^{(3)} (\x_4^{(3)})^2.$
We define the {\it loop (skew) Schur function} by
$$
s^{(r)}_{\rho/\nu}({\x}) = \sum_{T} \x^T
$$
where the summation is over all semistandard Young tableaux of
(skew) shape $\rho/\nu$.  We have the following analog of the Jacobi-Trudi formula.

\begin{thm} \cite[Theorem 7.6]{LP}
We have $s^{(r)}_{\lambda'/\mu'} = \det(e_{\lambda_i-\mu_j-i+j}^{(r-j+1+\mu_j)})$.
\end{thm}

\begin{prop} \label{prop:eh}
 For any $k>0$ we have $$e_0^{(r)} h_k^{(r-1)} - e_1^{(r-1)} h_{k-1}^{(r-2)} + e_2^{(r-2)} h_{k-2}^{(r-3)} - \ldots \pm e_k^{(r-k)} h_{0}^{(r-k-1)} = 0.$$
\end{prop}

\begin{proof}
 Let $\rho_i$ denote the hook shape $(k-i, 1, \ldots, 1)$ of size $k$. Then $e_i^{(r-i)} h_{k-i}^{(r-i-1)} = s_{\rho_{k-i}}^{(r-i-1)} + s_{\rho_{k-i+1}}^{(r-i)}$ is the sum of two loop Schur functions (one of which is zero if $i=0$ or $i=k$). This holds since any two terms in $e_i^{(r-i)}$ and $h_{k-i}^{(r-i-1)}$, viewed as tableaux of column and row shapes, fit together to give a semistandard tableau of one of the two hook shapes, depending on the entry in their smallest boxes. It is clear that as we sum over all $i$, all the hook shape loop Schur functions cancel out. 
\end{proof}

\subsection{Staircases}

For $k \geq 1$ and $r \in \Z/n\Z$, define
$$
\tau^{(r)}_k(\x_1,\x_2,\ldots,\x_m) = \sum_{I = \{i_1 \leq i_2 \leq \cdots \leq i_k\}} \x_{i_1}^{(r)} \x_{i_2}^{(r-1)} \cdots \x_{i_k}^{(r-k+1)}
$$
where the summation is over multisets $I \subset \{1,2,3,\ldots,m\}$ such that no number occurs more than $n-1$ times.  Note that if $k > m(n-1)$ we have $\tau^{(r)}_k = 0$.
It can be shown that $\tau^{(r)}_k$ lies in $\LSym_m$, but we shall not need it for what follows.

\begin{lem}
\label{L:tau}
We have
$$
\tau^{(r)}_k(\x_1,\ldots,\x_m) = \sum_{i=0}^\infty(-1)^i h^{(r)}_{k-in}e_i\left(\prod_{s\in \Z/n\Z} \x^{(s)}_1,\ldots,\prod_{s \in \Z/n\Z}\x^{(s)}_m\right),
$$
where the $e_i$ in the above formula denotes the usual elementary symmetric function.
\end{lem}

\begin{proof}
 Let $\x_{i_1}^{(r)} \x_{i_2}^{(r-1)} \cdots \x_{i_k}^{(r-k+1)}$ be a term in $h^{(r)}_{k}$. Let $J \subseteq I$ be the set of indexes which occur in $I = \{i_1 \leq i_2 \leq \cdots \leq i_k\}$ more than $n-1$ times. Then the coefficient of this term on the right is equal to $\sum_{K \subseteq J} (-1)^{|K|}$. This is equal to $1$ if $|J|=0$ and to $0$ otherwise. 
\end{proof}

\begin{lem}\label{L:taurecursion}
For each $k$, we have
$$
\sum_{i=0}^\infty (-1)^{i} e_i^{(r-i)}(\x_1,\ldots,\x_m)\tau_{k-i}^{(r-i-1)}(\x_1,\ldots,\x_m) = 0.
$$
\end{lem}

\begin{proof}
 Follows immidiately from Proposition \ref{prop:eh} and Lemma \ref{L:tau}.
\end{proof}

Define $\sigma^{(r)}_k(\x_1,\x_2,\ldots,\x_m) = \sum_{i=0}^k \x_1^{(r)}\x_1^{(r-1)}\cdots \x_1^{(r-i+1)} \tau^{(r-i)}_{k-i}(\x_2,\ldots,\x_m)$.

\begin{thm}\label{thm:stair}
For $m \geq 2$ and $r \in \Z/n\Z$, we have
\begin{align*}
&s^{(r)}_{(n-1)\delta_{m-1}}(\x_1,\x_2,\ldots,\x_m) \\&= \sigma^{(r)}_{(n-1)(m-1)}(\x_1,\ldots,\x_m) \sigma^{(r+1)}_{(n-1)(m-2)}(\x_2,\ldots,\x_{m})
\cdots
\sigma^{(r+m-2)}_{(n-1)}(\x_{m-1},\x_m).
\end{align*}
\end{thm}

\begin{remark}
When all the colors are identified, that is, $\x_i^{(s)} = \x_i^{(s')}$ for all $i$ and $s,s' \in \Z/n\Z$, 
Theorem \ref{thm:stair} is a coarsening of a result of Jucis \cite{Juc}, see also \cite[Ex. 7.30]{EC2}.
\end{remark}

It is clear that Theorem \ref{thm:stair} holds for $m =2$, for then it states that $s^{(r)}_{n-1}(\x_1,\x_2)=\sigma^{(r)}_{n-1}(\x_1,\x_2)$.  We shall prove Theorem \ref{thm:stair} in Section \ref{sec:proof}.

\section{Affine crystals}
\subsection{$R$-matrix}
We shall use \cite{Sh} as our main reference for affine crystals.

Recall that a {\it {Kirillov-Reshetikhin crystal}} of $\uqsln$ is the crystal graph corresponding to the highest weight module with highest weight proportional to one of the fundamental weights. An {\it {affine crystal}} is the tensor product of several Kirillov-Reshetikhin crystals.  We shall restrict our attention to the set $\c$ of affine crystals that are tensor products of symmetric powers of the standard representation.  Each element $b \in B$ of such a symmetric power can be identified with a single row semistandard tableau in the alphabet $1, \ldots, n$. 

If $B_1, B_2$ are Kirillov-Reshetikhin crystals, the {\it {combinatorial $R$-matrix}} is the unique isomorphism $R_{B_1,B_2}: B_1 \otimes B_2 \to B_2 \otimes B_1$ of affine crystals.  It is known that the combinatorial $R$-matrices generate an action of $S_m$ on $B_1 \otimes \cdots \otimes B_m$.


The combinatorial $R$-matrix has a convenient interpretation \cite{Sh} in terms of semistandard tableaux and the jeu de taquin algorithm \cite{EC2}. Let $b_1 \otimes b_2$ be an element of $B_1 \otimes B_2$.  Then $R_{B_1,B_2}(b_1 \otimes b_2) = c_1 \otimes c_2 \in B_2 \otimes B_1$ where $c_1, c_2$ are the unique pair of row shaped tableaux which jeu de taquin to the same tableau that $b_1$ and $b_2$ jeu de taquin to, as follows:
$$
\tableau[sY]{\bl&\bl&{1}&{2}&{2}&{4} \\{1}&{3}} \;\; \stackrel{R_{B_1,B_2}}{\longrightarrow} \;\; \tableau[sY]{ \bl&\bl&\bl&\bl&{2}&{4}\\ {1}&{1}&{2}&{3}} \;\; \text{since both jeu de taquin to} \;\; \tableau[sY]{{1}&{1}&{2}&{2}&{4} \\ {3}}
$$

The action of the combinatorial $R$-matrix can be explicitly described as follows (see \cite{HHIKTT}). Let $R_{B_1,B_2}(b_1 \otimes b_2) = c_1 \otimes c_2$ and let $\y^{(r)}_1$, $\y^{(r)}_2$, $s_1(\y^{(r)}_1)$, $s_1(\y^{(r)}_2)$ be the number of boxes filled with $r$-s in $b_1$,$b_2$,$c_1$,$c_2$ respectively, $r = 1, \ldots, n$. Then $$s_1(\y^{(r)}_1) = \y^{(r)}_2 + \ok_{r+1}(b_1,b_2) - \ok_{r}(b_1,b_2) \text{\;\;\;\;and\;\;\;\;} s_1(\y^{(r)}_2) = \y^{(r)}_1 + \ok_{r}(b_1,b_2) - \ok_{r+1}(b_1,b_2),$$ where $$\ok_{r}(b_1,b_2) = \min_{0 \leq s \leq n-1} (\sum_{t=1}^s \y^{(r+t-1)}_2+\sum_{t=s+1}^{n-1} \y^{(r+t)}_1)$$ and the indexes are taken in $\Z/n\Z$.

\begin{example}
In the example above $(\y^{(1)}_1,\y^{(2)}_1,\y^{(3)}_1,\y^{(4)}_1) = (1,0,1,0)$, $(\y^{(1)}_2,\y^{(2)}_2,\y^{(3)}_2,\y^{(4)}_2) = (1,2,0,1)$, $\ok_1(b_1, b_2) = \min(1,2,3,3)=1$, $\ok_2(b_1, b_2) = \min(2,3,3,3)=2$ and $s_1(\y^{(1)}_1) = 1 + 2 - 1 = 2$.
\end{example}

\subsection{Intrinsic energy function}
In \cite{KKMMNN} an important function $\H_{B,B'}: B \otimes B' \to \Z$ called {\it {local coenergy}} is defined for a tensor product of two affine crystals. If $B$ and $B'$ are Kirillov-Reshetikhin crystals, local coenergy has the following simple description in terms of tableaux \cite{Sh, KKN}. Given an element $b \otimes b'$ in $B \otimes B' \in \c$, form a two-row semistandard tableaux from $b$ and $b'$ as above. After that measure the maximal number of cells one can slide the top row to the left so that we still have a valid semistandard tableau. This maximal number of cells is the value of $\H_{B,B'}(b,b')$. 

\begin{example} For 
$$
\tableau[sY]{ \bl&\bl&\bl&\bl&{1}&{2}&{3}&{3}&{4} \\ {2}&{2}&{3}&{4}}
$$
the local coenergy is $3$, because 
$$
\tableau[sY]{ \bl \\ {2}} \tableau[sY]{{1}&{2}&{3}&{3}&{4} \\ {2}&{3}&{4}}\;\; \text{is semistandard while}
\;\;\;
\tableau[sY]{{1}&{2}&{3}&{3}&{4} \\ {2}&{2}&{3}&{4}} \;\;\text{is not.}
$$
\end{example}
It is easy to see that $\H_{B,B'}(b,b') = \ok_1(b,b')$, since each term $\sum_{t=1}^s \y^{(t)}_2+\sum_{t=s+1}^{n-1} \y^{(1+t)}_1$ is exactly the number of cells for which the boxes with $(s+1)$-s in them would allow to slide. The local coenergy is known to remain the same under the action of $R$-matrix: $\H_{B',B} = \H_{B,B'} \circ R_{B',B}$. 

We now define the {\it intrinsic energy function} $\D_B: B \to \Z$ (\cite{KKMMNN}) of an affine crystal $B \in \c$, following \cite{OSS}. For Kirillov-Reshetikhin crystals energy is zero.  Let $b = b_1 \otimes b_2 \otimes \ldots \otimes b_m$ be an element of an $m$-fold tensor product $B=B_1 \otimes B_2 \otimes \ldots \otimes B_m \in \c$.  We define the intrinsic energy $\D_B(b)$ to be 
\begin{equation}\label{E:global}
\D_B(b) = \sum_{1\leq i < j \leq m} \H_{B_i,B_j}(s_i s_{i+1} \cdots s_{j-2}(b_{j-1}) \otimes b_j).
\end{equation}
Although not obvious from this definition, intrinsic energy is preserved by the $R$-action. 
\begin{example}
Let us compute the intrinsic energy of the following element $b = b_1 \otimes b_2 \otimes b_3$.  
$$
\tableau[sY]{\bl&\bl&\bl&\bl&\bl&\bl&{1}&{2}&{3} \\ \bl&\bl& {1}&{2}&{2}&{4}\\ {1}&{3}} 
$$
We measure $\H(b_1, b_2) = 1$, $\H(b_2, b_3) = 2$. We also have seen above the result of applying $R_{B_1,B_2}$ to the first two tensor factors, which allows us to measure $\H(s_1(b_{2}), b_3) = 2$. Therefore $\D_B(b) = 1+2+2=5$.
$$
\tableau[sY]{\bl \\ {1}} \tableau[sY]{{1}&{2}&{2}&{4} \\ {3}} \;\;\;\;\;\;\; \tableau[sY]{ \bl \\ {1}&{2}} \tableau[sY]{{1}&{2}&{3} \\ {2}&{4}} \;\;\;\;\;\;\;  \tableau[sY]{{1}&{2}&{3} \\ {2}&{4}}
$$
\end{example}

\subsection{Product (summation) formula for intrinsic energy}
In this section, we switch from piecewise linear functions to rational functions.  The two worlds are connected via tropicalization: if $f$ is a subtraction-free polynomial, we let $\trop(f)$ denote the tropicalization of $f$, obtained by replacing addition by minimum, and multiplication by addition.

We are given a rectangular array of variables $\z_{i}^{(r)}$, $i = 1, \ldots, m$, $r \in \Z/n\Z$, with columns $b_i = (\z_i^{(1)}, \z_i^{(2)}, \ldots, \z_i^{(n)})$. It is very convenient to make the following change of variables: $\x_i^{(r)} = \z_i^{(r+1-i)}$. Define $$\k_{r}(b_j,b_{j+1}) = \sum_{s=0}^{n-1} (\prod_{t=1}^s \x^{(r+t)}_{j+1} \prod_{t=s+1}^{n-1} \x^{(r+t)}_j),$$ so that $\trop(\k_r(b_j,b_{j+1})) = \ok_{r-j+1}(b_j,b_{j+1})$.

In the variables $\x_i^{(r)}$, the birational $R$-matrix acts (see \cite[Proposition 3.1]{Y}\footnote{Our variables $\z_i^{(r)}$ are nearly the same as Yamada's $x_r^i$, differing by a reversal of the orientation of the circle.}) via algebra isomorphisms $s_1,s_2,\ldots,s_{m-1}$ of the field of rational functions in $\{\x_i^{(r)}\}$, given by 
$$s_j(\x^{(r)}_j) = \frac{\x^{(r+1)}_{j+1}  \k_{r+1}(b_j,b_{j+1})}{\k_{r}(b_j,b_{j+1})} \text{\;\;\;\; and \;\;\;\;} s_j(\x^{(r)}_{j+1}) = \frac{\x^{(r-1)}_j \k_{r-1}(b_j,b_{j+1})}{\k_{r}(b_j,b_{j+1})}$$
and $s_j(\x^{(r)}_k)=\x^{(r)}_k$ for $k \neq j,j+1$.  We also have $\HH(b_j \otimes b_{j+1})= \kappa_j(b_j, b_{j+1})$, the rational analogue of the local coenergy.  We let $\DD_B$
denote the rational analogue of the intrinsic energy function, so that $\trop(\DD_B) = \D_B$.  The main result of this section is a product formula for $\DD_B$.  We remark that Kirillov \cite{K} has also studied the rational functions $\HH$ and $\DD_B$.

\begin{lem} \label{lem:tact}
 Suppose $1 \leq i < j \leq m$. 
Then
$$\k_r(s_is_{i+1}\cdots s_{j-2}(b_{j-1}),b_j) = \frac{\sigma^{(r-j+i)}_{(n-1)(j-i)}(\x_i,\x_{i+1},\ldots,\x_j)}{\sigma^{(r-j+i)}_{(n-1)(j-i-1)}(\x_i,\x_{i+1},\ldots,\x_{j-1})}$$ and 
 $$s_is_{i+1}\cdots s_{j-1}(\x^{(r)}_j) = \frac{\x^{(r-j+i)}_i \sigma^{(r-j+i-1)}_{(n-1)(j-i)}(\x_i,\x_{i+1},\ldots,\x_j)}{\sigma^{(r-j+i)}_{(n-1)(j-i)}(\x_i,\x_{i+1},\ldots,\x_j)} = \frac{\sigma^{(r-j+i)}_{(n-1)(j-i)+1}(\x_i,\x_{i+1},\ldots,\x_j)}{\sigma^{(r-j+i)}_{(n-1)(j-i)}(\x_i,\x_{i+1},\ldots,\x_j)}.$$
\end{lem}

\begin{proof}
 We prove the two statements in parallel by induction on $j-i$. For $j-i=1$ they coincide with the formulae for the $\k_r$ and the $R$-action of $s_i$.  By the induction assumption 
$$s_is_{i+1}\cdots s_{j-2}(\x^{(r)}_{j-1}) =  \frac{\x^{(r-j+i+1)}_i \sigma^{(r-j+i)}_{(n-1)(j-i-1)}(\x_i,\x_{i+1},\ldots,\x_{j-1})}{\sigma^{(r-j+i+1)}_{(n-1)(j-1-i)}(\x_i,\x_{i+1},\ldots,\x_{j-1})}.$$ 
Therefore 
\begin{align*}
&\k_r(s_is_{i+1}\cdots s_{j-2}(b_{j-1}),b_j) \\
&= \sum_{s=0}^{n-1} (\prod_{t=1}^s \x^{(r+t)}_{j} \prod_{t=s+1}^{n-1} s_is_{i+1}\cdots s_{j-2}(\x^{(r+t)}_{j-1}))\\ 
&=\sum_{s=0}^{n-1} \frac{\prod_{t=s-j+i+2}^{n-j+i} \x_i^{(r+t)} \sigma^{(r+s-j+i+1)}_{(n-1)(j-i-1)}(\x_i,\x_{i+1},\ldots,\x_{j-1}) \prod_{t=1}^s \x_j^{(r+t)}} {\sigma^{(r-j+i)}_{(n-1)(j-i-1)}(\x_i,\x_{i+1},\ldots,\x_{j-1})}\\ &=\frac{\sigma^{(r-j+i)}_{(n-1)(j-i)}(\x_i,\x_{i+1},\ldots,\x_j)}{\sigma^{(r-j+i)}_{(n-1)(j-i-1)}(\x_i,\x_{i+1},\ldots,\x_{j-1})}.
\end{align*} 
The last equality holds because in a term of $\sigma^{(r-j+i)}_{(n-1)(j-i)}(\x_i,\x_{i+1},\ldots,\x_j)$ the number of the $\x_j^{(t)}$-s is at most $n-1$, while the number of the $\x_i^{(t)}$-s and the $\x_j^{(t)}$-s together should be at least $n-1$.

Now we can also prove the second claim, since $$s_is_{i+1}\cdots s_{j-1}(\x^{(r)}_j) = \frac{s_is_{i+1}\cdots s_{j-2}(\x^{(r-1)}_{j-1}) \k_{r-1}(s_is_{i+1}\cdots s_{j-2}(b_{j-1}),b_j)}{\k_{r}(s_is_{i+1}\cdots s_{j-2}(b_{j-1}),b_j)}=$$ $$\frac{\x^{(r-j+i)}_i \sigma^{(r-j+i-1)}_{(n-1)(j-i-1)}(\x_i,\ldots,\x_{j-1})}{\sigma^{(r-j+i)}_{(n-1)(j-1-i)}(\x_i,\ldots,\x_{j-1})} \frac{\sigma^{(r-j+i-1)}_{(n-1)(j-i)}(\x_i,\ldots,\x_j)}{\sigma^{(r-j+i-1)}_{(n-1)(j-i-1)}(\x_i,\ldots,\x_{j-1})} \frac {\sigma^{(r-j+i)}_{(n-1)(j-i-1)}(\x_i,\ldots,\x_{j-1})}{\sigma^{(r-j+i)}_{(n-1)(j-i)}(\x_i,\ldots,\x_j)}$$ $$= \frac{\x^{(r-j+i)}_i \sigma^{(r-j+i-1)}_{(n-1)(j-i)}(\x_i,\ldots,\x_j)}{\sigma^{(r-j+i)}_{(n-1)(j-i)}(\x_i,\ldots,\x_j)}.$$
\end{proof}

\begin{thm} \label{thm:energyprod}
 We have $$\DD_B(b) = \sigma^{(n)}_{(n-1)(m-1)}(\x_1,\ldots,\x_m) \sigma^{(1)}_{(n-1)(m-2)}(\x_2,\ldots,\x_{m})
\cdots
\sigma^{(m-2)}_{(n-1)}(\x_{m-1},\x_m).$$
\end{thm}

\begin{proof}
The result follows from Lemma \ref{lem:tact} and \eqref{E:global}.
\end{proof}

\begin{remark}
Comparing \cite[Theorem 4.2]{NY2} with Lemma \ref{lem:tact} and Theorem \ref{thm:energyprod} one can see that the tropicalization of the $\sigma^{(i)}_{(n-1)(m-1-i)}(\x_{i+1},\ldots,\x_{m})$ is essentially the index $\ind(m-i)$ in the index decomposition of charge in \cite{NY2}. 
\end{remark}
\begin{remark}\label{rem:irred}
Theorem \ref{thm:energyprod} gives the irreducible factorization of $\DD_B(b)$ (whereas Theorem \ref{thm:main2} gives the monomial expansion).  To see this, one first notes that $\sigma^{(r)}_{(n-1)}(\x_1,\x_2)$ has a unique monomial which involves $\x_2^{(r)}$, from which one deduces the irreducibility.  Now suppose that $\sigma^{(r)}_{(n-1)(m-1)}(\x_1,\x_2,\ldots,\x_m)$ factorizes non-trivially as the product $fg$.  We may write $f= a\x_m^{(r-(n-1)(m-2))}+b$ and $g=c$ as polynomials in $\x_m^{(r-(n-1)(m-2))}$, where $a,b,c$ do not involve $\x_m^{(r-(n-1)(m-2))}$.  One verifies that none of the variables $\x_m^{(s)}$ divide $\sigma^{(r)}_{(n-1)(m-1)}(\x_1,\x_2,\ldots,\x_m)$ and every monomial which contains $\x_m^{(r-(n-1)(m-2))}$ is divisible by the product $\x_m^{(r-(n-1)(m-2))}\x_m^{(r-(n-1)(m-2)-1)} \cdots \x_m^{(r-(n-1)(m-1)+1)}$.  Thus we have $f = a' \x_m^{(r-(n-1)(m-2))}\x_m^{(r-(n-1)(m-2)-1)} \cdots \x_m^{(r-(n-1)(m-1)+1)}+ b$, where $a'$ is a polynomial not involving any $\x_m^{(s)}$.  It is easy to see that $a'$ cannot be a unit.  But we then have a non-trivial factorization $a'c = \sigma^{(r)}_{(n-1)(m-2)}(\x_1,\ldots,\x_{m-1})$, and we may proceed by induction.
\end{remark}


\section{Proof of Theorem \ref{thm:stair}}\label{sec:proof}
We let $A_m$ denote the Jacobi-Trudi matrix for the dilated staircase Schur funtion $s^{(r)}_{(n-1)\delta_{m-1}}$.  By adding extra columns of size $0$ to $(n-1)\delta_{m-1}$ we may assume that $A_m$ is a $na \times na$ matrix.  (Specifically, $a = \lceil(n-1)(m-1)/n\rceil$.)

\begin{example}
For $n = 3$ we have
$$
A_4 =\left( \begin{array}{cccccc} 
e_3^{(r)} & e_4^{(r-1)} &&&&\\
e_2^{(r)} & e_3^{(r-1)} & e_4^{(r-2)} &&& \\
e_0^{(r)} & e_1^{(r-1)} & e_2^{(r-2)} & e_3^{(r)} & e_4^{(r-1)}&\\
&e_0^{(r-1)} & e_1^{(r-2)} & e_2^{(r)} & e_3^{(r-1)} & e_4^{(r-2)}\\
& &  & e_0^{(r)} & e_1^{(r-1)} & e_2^{(r-2)}\\
&&  & & e_0^{(r-1)} & e_1^{(r-2)}
\end{array} \right) 
$$
\end{example}

\begin{lem}\label{lem:translate}
Suppose $n < i \leq na$.  Then column $i$ of $A_m$ is obtained from column $i-n$ by shifting the non-zero entries down by $n-1$.
\end{lem}




Let $B_m$ denote the $n(a+1)-1\times n(a+1)$ matrix obtained by adding
$n$ extra columns to $A_m$, so that Lemma \ref{lem:translate} is still true.
\begin{example}
For $n = 3$ we have
$$
B_4 =\left( \begin{array}{ccccccccc} 
e_3^{(r)} & e_4^{(r-1)} &&&&&&&\\
e_2^{(r)} & e_3^{(r-1)} & e_4^{(r-2)} &&&&&& \\
e_0^{(r)} & e_1^{(r-1)} & e_2^{(r-2)} & e_3^{(r)} & e_4^{(r-1)}&&&&\\
&e_0^{(r-1)} & e_1^{(r-2)} & e_2^{(r)} & e_3^{(r-1)} & e_4^{(r-2)}&&&\\
&&& e_0^{(r)} & e_1^{(r-1)} & e_2^{(r-2)} & e_3^{(r)} & e_4^{(r-1)} &\\
&&&& e_0^{(r-1)} & e_1^{(r-2)} & e_2^{(r)} & e_3^{(r-1)} & e_4^{(r-2)} \\
&&&&&& e_0^{(r)} & e_1^{(r-1)} & e_2^{(r-2)}\\
&&&&&&& e_0^{(r-1)} & e_1^{(r-2)}
\end{array} \right) 
$$
\end{example}

Let $$\T = (\tau^{(r-1)}_{(n-1)m},-\tau^{(r-2)}_{(n-1)m-1},\ldots,\pm\tau^{(r-n(a+1)+1)}_{(n-1)m-n(a+1)+1})$$ be a column vector with components in $\LSym$.
\begin{prop}\label{P:Bzero}
The vector $B_m \cdot \T$ is the zero vector.
\end{prop}
\begin{proof}
It follows immediately from Lemma \ref{L:taurecursion}, and the fact that $\tau^{(s)}_{k} = 0$ as long as $k > (n-1)m$.
\end{proof}

Let $B_{i,m}$ be the square matrix obtained from $B_m$ by removing the $i$-th column.  It is easy to see that $\det(B_{i,m}) = s^{(r-1)}_{(n-1)\delta_{m}/(i-1)}$ is a loop skew Schur function.

\begin{prop}\label{P:detBim}
We have $\det(B_{i,m}) = \tau^{(r-i)}_{(n-1)m-i+1}\, \det(A_m)$.
\end{prop}
\begin{proof}
Since $\det(B_{i,m}) = s^{(r-1)}_{(n-1)\delta_{m}/(i-1)}$ is never the zero polynomial in the $e_i^{(s)}$, we deduce that considered as a matrix with ceofficients in the field ${\rm Frac} (\x_i^{(s)})$, the matrix $B_m$ has maximal rank.  There is thus, up to scaling, a unique solution to the equation $B_m \cdot v = 0$.  By expanding the determinant of the matrix obtained from $B_m$ by repeating a row, it is easy to see that $v = (\det(B_{1,m}),-\det(B_{2,m}),\ldots,\pm\det(B_{na+n,m}))$ is a solution.  But by Proposition \ref{P:Bzero} so is $\T$. Thus the two vectors are proportional, and it remains to check that the scaling coefficient is $\det(A_m)$. We have $\det(B_{1,m}) = e_m^{(r-1)} \dotsc e_m^{(r-n+1)} \det(A_m)$, while at the same time $\tau^{(r-1)}_{(n-1)m} = e_m^{(r-1)} \dotsc e_m^{(r-n+1)}$, and the statement follows.
\end{proof}

\begin{proof}[Proof of Theorem \ref{thm:stair}]
We have already verified the case $m= 2$, so we suppose that $m >2$, and by induction on $m$ that
\begin{align*}
&s^{(s+1)}_{(n-1)\delta_{m-1}}(\x_2,\x_3,\ldots,\x_{m+1}) \\
&= \sigma^{(s+1)}_{(n-1)(m-1)}(\x_2,\ldots,\x_{m+1}) \sigma^{(s+2)}_{(n-1)(m-2)}(\x_3,\ldots,\x_{m+1})
\cdots
\sigma^{(s+m-1)}_{(n-1)}(\x_{m},\x_{m+1})
\end{align*}
for every $s \in \Z/n\Z$.  We calculate that
\begin{align*}
&s^{(r-1)}_{(n-1)\delta_{m}}(\x_1,\x_2,\ldots,\x_{m+1}) \\
&= \sum_{i=0}^{(n-1)m} \x_1^{(r-1)}\x_1^{(r-2)}\cdots \x_1^{(r-i)} s^{(r-1)}_{(n-1)\delta_{m}/(i)}(\x_2,\x_3,\ldots,\x_{m+1}) \\
&= \sum_{i=0}^{(n-1)m} \x_1^{(r-1)}\x_1^{(r-2)}\cdots \x_1^{(r-i)} \det(B_{i,m})(\x_2,\x_3,\ldots,\x_{m+1}) \\
&= \left(\sum_{i=0}^{(n-1)m} \x_1^{(r-1)}\x_1^{(r-2)}\cdots \x_1^{(r-i)} \tau^{(r-i-1)}_{(n-1)m-i+1}(\x_2,\x_3,\ldots,\x_{m+1})\right) s^{(r)}_{(n-1)\delta_{m-1}}(\x_2,\x_3,\ldots,\x_{m+1})\\
&= \sigma^{(r-1)}_{(n-1)m}(\x_1,\x_2,\ldots,\x_{m+1})s^{(r)}_{(n-1)\delta_{m-1}}(\x_2,\x_3,\ldots,\x_{m+1})
\end{align*}
where in the first equality we used the tableau definition of $s^{(r)}_{(n-1)\delta_{m}}$, and in the penultimate equality we 
used Proposition \ref{P:detBim}.
\end{proof}

\begin{remark}
 It is clear from the theory developed in \cite{LP} that loop Schur functions are invariants of the action of the symmetric group $S_m$ via the birational $R$-action. Thus we have demonstrated directly that the energy function is an invariant of this action. This property is not obvious from the definition we use.
\end{remark}


\end{document}